    \documentclass[final,1p,times]{elsarticle}

\usepackage{amsmath, amsthm, amscd, amsfonts, amssymb, graphicx, color}
\usepackage[bookmarksnumbered, colorlinks, plainpages]{hyperref}
\usepackage{wrapfig}
\usepackage{subfigure}

\newdefinition{rmk}{Remark}
\newproof{pf}{Proof}

\journal{}

\begin{document}

\newcommand{\be}{\begin{equation}}
\newcommand{\ee}{\end{equation}}
\newcommand{\bt}{\beta}
\newcommand{\ep}{\epsilon}
\newcommand{\al}{\alpha}
\newcommand{\laa}{\lambda_\alpha}
\newcommand{\lab}{\lambda_\beta}
\newcommand{\no}{|\Omega|}
\newcommand{\nd}{|D|}
\newcommand{\Om}{\Omega}
\newcommand{\h}{H^1_0(\Omega)}
\newcommand{\lt}{L^2(\Omega)}
\newcommand{\la}{\lambda}
\newcommand{\ro}{\varrho}
\newcommand{\cd}{\chi_{D}}
\newcommand{\cdc}{\chi_{D^c}}
\theoremstyle{plain}
\newtheorem{thm}{Theorem}[section]
\newtheorem{cor}[thm]{Corollary}
\newtheorem{lem}[thm]{Lemma}
\newtheorem{prop}[thm]{Proposition}
\theoremstyle{definition}
\newtheorem{defn}{Definition}[section]
\newtheorem{exam}{Example}[section]
\theoremstyle{remark}
\newtheorem{rem}{Remark}[section]
\numberwithin{equation}{section}
\renewcommand{\theequation}{\thesection.\arabic{equation}}
\numberwithin{equation}{section}
\begin{frontmatter}



\title{ Extremal energies of Laplacian operator: Different configurations for steady vortices  }
\author[am]{Seyyed Abbas Mohammadi }
\ead{mohammadi@yu.ac.ir}

%
%
\address[am]{Department of Mathematics, College of Sciences,
Yasouj University, Yasouj, Iran, 75918-74934 }

%
%
%
\begin{abstract}
In this paper, we study a maximization and a minimization problem associated with a Poisson boundary
value problem. Optimal solutions in a set of rearrangements of a given function define stationary and stable flows
of an ideal fluid in two dimensions. The main contribution of this paper is  to determine the optimal solutions. At first,
we derive the solutions analytically when the problems are in  low contrast regime. Moreover,
it is established that the  solutions of both problems  are unique.       
Secondly, for  the high contrast regime, two optimization algorithms are developed. For the minimization problem, we prove that  our algorithm converges to the global minimizer regardless of the initializer. The maximization algorithm is capable of deriving all local maximizers including the global one. Numerical experiments  leads us to a
conjecture about  the location of the maximizers in the set of rearrangements of a function.
\end{abstract}
\begin{keyword}
 Laplacian Operator\sep Shape optimization \sep Analytic solution  \sep Rearrangement
\sep Steady vortices


\MSC  49Q10 \sep 35J25 \sep 76A02  \sep 65Z05 \sep 65K10  

\end{keyword}

\end{frontmatter}


\section{Introduction}\label{intro}
 
Optimal shape design for energy functionals corresponding to elliptic operators provides
a vast number of interesting and challenging mathematical problems; see \cite{henrot} and the references.
This class of problems arise naturally in many different fields, such as mechanical vibrations, electromagnetic cavities, photonic crystals and population dynamics.

 In this paper we are considering the problem of optimizing an energy functional corresponding
 to the Laplacian operator with Dirichlet's boundary conditions.


 Let $\Omega$ be a bounded smooth domain  in $\mathbb{R}^N$ and let $f_0=\al \chi_{D_0}+\beta \chi_{D_0^c} $ be a step function such that $D_0\subset \Om$ and  $|D_0|=A>0$ and $\al> \bt>0$. Notation $|.|$ stands for Lebesgue measure. Define $\mathcal{F}$  as
the family of all measurable functions which are rearrangement of $f_0$, we will see later in section \ref{prem} that $f\in \mathcal{F} $ if and only if $f=\al \chi_{D}+\beta \chi_{D^c} $ where $D\subset \Om, \: |D|=A $. For $f \in \mathcal{F}$, consider the following Poisson boundary value problem

\begin{equation} \label{mpde}
-\Delta u=f \quad
\mathrm{in}~\Omega,\qquad \quad u=0 \quad \mathrm{on}~\partial
\Omega.
\end{equation}
 This equation has many applications in different areas of engineering and science including electrostatics, steady fluid flow and Brownian motion to name a few \cite{strauss}. As an application,
  consider the planar motion under
  an irrotational body force of an incompressible, inviscid fluid contained in an infinite cylinder of uniform cross-section $\Om$. If $f$ represents the vorticity for the flow of the
  fluid, then $u_f$ is the stream function. In equation \eqref{mpde}, flows are sought in which the vorticity is a rearrangement of a prescribed function $f_0$.
  
   Defining the energy functional $$\Psi(f)=\int_{\Om} f u_fdx,$$
   we are interested in the following optimization problems
   \begin{equation}\label{maxp}
  \underset{f\in \mathcal{F}}{\max } \Psi(f),
   \end{equation}
    \begin{equation}\label{minp}
    \underset{f\in \mathcal{F}}{\min } \Psi(f).
    \end{equation}
    Indeed, one can recast optimization problems \eqref{maxp} and \eqref{minp} into the problem of finding a set $D\subset \Omega,  \: |D|=A$ where the corresponding  vorticity function $f=\al \chi_{D}+\beta \chi_{D^c}$ is the optimal solution of the functional 
    \begin{equation*}
    \Psi(D)=\int_{\Omega} \left(\al \chi_{D}+\beta \chi_{D^c}\right) u(D)dx,
        \end{equation*}
         where $u=u(D)$ is the corresponding solution of \eqref{mpde}. Such set $D$ is called an optimal set.
         
          Physically, the functional $\Psi(f)$ 
          represents the kinetic energy of the fluid. Steady flows correspond to stationary points of the kinetic
          energy relative to rearrangement of the vorticity. A vorticity $f$ that maximize or minimize the energy functional $\Psi(f)$
          in a set of rearrangements therefore defines a stationary and stable flows. These solutions determine different configurations of a
          region of vorticity in an otherwise irrotational flow \cite{bu89, bahrami}.

          There is a plenitude of papers studied rearrangement optimization problems \eqref{maxp} and \eqref{minp},
          see  \cite{bu89, bu87, bu91, master,Alvino} and  the references therein.
                   These studies have investigated existence, uniqueness and some qualitative properties of the solutions for  \eqref{maxp} and \eqref{minp} such as symmetry of the maximizers or minimizers . It has been proved that the minimization problem \eqref{minp}  has a unique solution and the optimal solution of the problem has been determined when $\Omega$ is a ball centered at the origin. Although the maximization problem \eqref{maxp} admits a solution for general domain $\Omega$, the uniqueness of the solution has been established only for the case that $\Omega$ is a ball. Indeed, solutions of \eqref{maxp} are not unique in general.
          
           In  optimization problems like \eqref{maxp} and \eqref{minp}, one of challenging mathematical problems after the problem of existence is an exact formula of the optimizer or optimal shape design. Most papers in this field answered this question just in case $\Om$ is a ball. For other domains qualitative properties of solutions were investigated and some partial answers were given \cite{chanillo, abbasali,abbasali2,derlet}. From the physical point of view, it is important  to know the shape of the optimal vorticity  in case $\Om$ is not a ball.
           This class of problems is difficult to solve because of the lack of the topology information of the optimal shape.
           
           The main contribution of this paper is to determine  the optimal shape design for optimization problems \eqref{maxp} and \eqref{minp} when  $\Omega$ is an arbitrary  domain .
           We will find the solution to problems \eqref{maxp} and \eqref{minp} analytically when $\al$ and $\bt$ are  close to each other which are called the low contrast regime problems.
           Although it has been proved that solutions of \eqref{maxp}
           are not unique in general, we  establish that the  solutions of both problems \eqref{maxp} and \eqref{minp} are unique when $\alpha$ and $\beta$ are close to each other.       
                       The analytical solutions will be found by expanding the energy functional $\Psi$ with respect to $(\alpha - \beta)$.  A similar problem in the low contrast regime has been investigated by  Conca  \emph{et al} in \cite{Conca3} in order to minimize the first eigenvalue of an elliptic operator with Dirichlet conditions in a ball.
           
            When $\al$ and $\bt$ are not close to each other, the high contrast regime, there must be numerical approaches to determine the optimal shape design.
           The mostly used  methods now are the homogenization method \cite{alliare} and the level set            method \cite{osher}. The level set method is well known for its capability  to handle topological  changes, such as breaking one component into several, merging several components into one and forming sharp corners. This approach has been applied to the study of extremum problems of eigenvalues of inhomogeneous
           structures including the identification of composite membranes with extremum eigenvalues \cite{oshersantosa}, design of composite  materials with a desired spectral gap or maximal spectral gap \cite{kaoyab},
           finding optical devices that have a high quality factor \cite{kaosantosa2} and principle eigenvalue optimization in population biology \cite{kaoluo}.
            Recently, Kao and Su \cite{kao2} proposed an efficient rearrangement algorithm based on the Rayleigh quotient formulation of eigenvalues. They have solved minimization and maximization problem for the $k$-th eigenvalue ($k\geq 1$) and maximization of spectrum ratios of the second order elliptic differential operator in $\mathbb{R}^2$.
            
             Motivated by Kao and Su's method, two optimization algorithms is developed in order to find the optimal energies for problems \eqref{maxp} and \eqref{minp}. For  minimization problem \eqref{minp}, we prove that our algorithm converges to the global minimizer of \eqref{minp} regardless of the initializer. It is worth noting that in \cite{kao2} the numerical simulations  have been repeated with some different initial guesses to find the optimal solutions. Furthermore, an acceptance-rejection method have been included for some algorithms in \cite{kao2}. This method which is called the partial swapping method increases the cost of computations. This is due to the fact that one must accept or reject some data by checking whether the objective function is improved or not. In our algorithm for minimization, we replace the partial swapping method with a step that one should only verify that some of data fulfill a criterion. 
             For the maximization problem \eqref{maxp}, we show that our algorithm converges to a local maximizer.
                      Running the algorithm with different initializers, one can obtain the global maximizer. It is worth noting here that local maximizers are also corresponding to steady flows of the fluid and determining them are physically important \cite{bu89}.  Employing our maximization algorithm, one can derive  local maximizers of  \eqref{maxp} for complicated domains such as   domain $\Omega$ that approximates the union of $n$ balls.    Numerical experiments  lead us to a conjecture   that the local maximizers are in the farthest away from $\hat{f}$, the global minimizer of  \eqref{minp}, relative to $\mathcal{F}$.
                      If one starts the maximization algorithm from a function farthest away from $\hat{f}$ then the maximization algorithm converges faster.


\section{Preminileries}\label{prem}

In this section we state some results from the rearrangement theory related to our optimization problems \eqref{maxp}- \eqref{minp}. The reader
can refer to  \cite{bu87, bu89, Alvino} which are  standard references for  the rearrangement theory.

 Throughout this paper we shall write increasing instead  of non- decreasing, and decreasing instead of non- increasing.
 
 \begin{defn}\label{readef}
 	Two Lebesgue measurable functions  $f: \Om \rightarrow \Bbb{R}$, $f_0:\Om \rightarrow \Bbb{R}$, are said to be rearrangements of each other if
 	\begin{equation}\label{rea}
 	|\{x\in \Om : f(x)\geq r \}|=|\{x\in \Om : f_0(x)\geq r\}|\qquad~\quad\forall r \in \mathbb{R}.
 	\end{equation}
 \end{defn}
 The notation $f\sim f_0$ means that $f$ and $f_0$ are rearrangements of each other. Consider $f_0:\Om \rightarrow \Bbb{R}$, the class of rearrangements generated by $f_{0}$, denoted $\mathcal{F}$, is defined as follows\\
 \begin{equation*}
 \mathcal{F}=\{f:f\sim f_{0}\}.
 \end{equation*}
 The closure of $\mathcal{F}$ in $\lt$ with respect to the weak topology is denoted by $\bar{\mathcal{F}}$.
 
  Set $f_0(x) = \al \chi_{D_0}+ \bt  \chi_{D_0^c}$ with  $|D|=A$  then we have the following
  technical assertion from \cite{mohammadi-bilap}.
 
 \begin{lem}\label{chirho}
 	Function $f$ belongs to the rearrangement class $\mathcal{F}$ if and only if  $f= \al \chi_{D}+ \bt  \chi_{D^c}$ where $D$ is a subset of $\Om$ with $|D|=A$.
 \end{lem}
 
 Next two lemmas provide our main tool for deducing the analytical and  numerical results.
 
 \begin{lem}\label{bathmax}
 	Let $u(x)$ be a nonnegative function in $L^1(\Om)$ such that its level sets have measure zero. Then the maximization problem
 	\begin{equation}\label{bathtubsup}
 	\sup_{f \in \bar{\mathcal{F}}} \int_{\Om}fu dx,
 	\end{equation}
 	is uniquely solvable by $\widehat{f}(x)=\al \chi_{\hat{D}}+\bt  \chi_{\hat{D}^c}$ where $|\widehat{D}|=A$ and
 	\begin{equation*}
 	\widehat{D}=\{x\in \Om:\,\:u(x)\geq t\},
 	\end{equation*}
 	\begin{align*}
 	t=\sup\{s\in \mathbb{R} : |\{x\in \Om:\,\:u(x)\geq s\}|\geq A\}.
 	\end{align*}
  \end{lem}

\begin{lem}\label{bathmin}
	Let $u(x)$ be a nonnegative function in $L^1(\Om)$ such that its level sets have measure zero. Then the minimization problem
	\begin{equation}\label{bathtubinf}
	\inf_{f \in \bar{\mathcal{F}}} \int_{\Om}fu dx,
	\end{equation}
	is uniquely solvable by $\widehat{f}(x)=\al \chi_{\hat{D}}+\bt  \chi_{\hat{D}^c}$ where $|\widehat{D}|=A$ and
	\begin{equation*}
		\widehat{D}=\{x\in \Om:\,\:u(x)\leq t\},
	\end{equation*}
	\begin{align*}
		t=\inf\{s\in \mathbb{R} : |\{x\in \Om:\,\:u(x)\leq s\}|\geq A\}.
	\end{align*}

\end{lem}
We gain some insight into the solution of problem \eqref{mpde} from the following lemma.
\begin{lem}\label{uf}
	Let $f \in L^P(\Omega), p>1,$ be a nonnegative function and let $u(x)$ be a solution to problem \eqref{mpde} with  the right hand side $f$. Then 
	\\
	i) $u\in C(\bar{\Om})\cap W^{2,p}(\Om) \cap  W^{1,p}_0(\Om)$  and $\|u\|_{W^{2,p}(\Om)}\leq C \|f\|_ {L^P(\Omega)}$ where
	$C$ is independent of $f$ and $u$,\\
	ii) $0\leq u\leq\underset{\Omega}{\sup} f(x)$,\\
	iii) If $p\geq N $ then we have 
	$$\underset{\Omega}{\sup}\: u(x)\leq \frac{d}{N \omega_N^{\frac{1}{N}}}\|f\|_{ L^N(\Omega)},$$
	where d is the diameter of $\Omega$ and $\omega_N$ denotes the volume of the unit sphere in $\mathbb{R}^N$.\\
iv) If $f$ is a positive function, then $0<u$ in $\Omega$ and level sets of $u$ have measure zero.
\end{lem}
\begin{proof}
	The  proof of i) follows from Theorem 9.15 and Lemma 9.17 of \cite{gilbarg}. We have part ii) applying
	Theorem 9.27 of \cite{brezis}.
 The proof of part iii) is deduced from Theorem 9.1 and Lemma 9.3 of \cite{gilbarg}.
 The proof of the last part follows from Theorem 8.20 and Lemma 7.7 in \cite{gilbarg}.
 
\end{proof}

\section{Low Contrast Regime}\label{lowsec}

 This section is concerned with the optimization problems \eqref{maxp} and \eqref{minp} in a low contrast regime. This means that $\alpha$ and $\beta$ are close to each other: $\alpha= \beta+ \epsilon$ with $\epsilon>0$ small. Then a function $f$ in $\mathcal{F}$ has the form $f=\beta+ \epsilon \chi_D$, in this case. We will determine the solutions to  \eqref{maxp} and \eqref{minp} analytically when $\epsilon>0$ is small enough. In order to show the dependence on $\epsilon$ and $D$, we use $u=u^{\epsilon}(D)$ for the solution of \eqref{mpde} and 
 \begin{equation}\label{psiep}
 \Psi=\Psi(\epsilon,D)=\int_{\Om} (\bt+\epsilon \chi_D) u^\epsilon(D)dx.
 \end{equation}
 
   Let $\phi_0$ and $\phi_1(D)$ be the solutions of \eqref{mpde}  with the right hand sides $\bt$ and $\chi_D$ respectively where $D$ is a subset of $\Omega$ with $|D|=A.$  Recall that $\phi_0$ is a positive function and its
   level sets have measured zero in view of Lemma \ref{uf} -(iv).
   According to Lemma \ref{bathmax}, there is a set $D_M\subset \Omega$ which is the  solution  of the following optimization problem
  \begin{equation}\label{dmmaker}
  \underset{D \subset \Om,\: |D|=A}{\sup}\int_{\Om} \chi_D \phi_0dx.
  \end{equation}
  Note that $D_M $ is uniquely determined by
  \begin{equation}\label{tmform}
  D_M=\{x\in \Om:\,\:\phi_0(x)\geq t_M\},\:\: t_M=\sup\{s\in \mathbb{R} : |\{x\in \Om:\,\:\phi_0(x)\geq s\}|\geq A\}.
      \end{equation}
      
     We can now state the main result of this section.
   \begin{thm}\label{lowmaxthem}
   There exists $\ep_0>0$ such that for all $D\subset \Om,\: |D|=A,$ we have
   $$\Psi(\ep,D_M)\geq \Psi (\ep,D),\quad for \:\: all \:\: 0<\ep\leq\ep_0,$$
   and the equality occurs only  when $D=D_M$ almost everywhere in $\Om$.
   \end{thm}
   \begin{proof}
   For the sake of clarity we divide the proof into several steps.  Let $\ep>0$ be a constant which is small compared to $1$.
   
   Step $1$. In view  of the linearity of the Laplace operator, we have 
   $u^\ep(D)=\phi_0+\ep \phi_1(D)$
   and so 
   \begin{equation}\label{psiepform}
    \Psi(\ep, D)=\int_{\Om} (\bt+\ep \chi_D)(\phi_0+\ep \phi_1(D))dx=\int_\Om \bt \phi_0dx+2\ep\int_\Om \chi_D \phi_0 dx+ \ep^2 \int_\Om \chi_D \phi_1(D) dx,
   \end{equation}
      using
   $$\bt\int_\Om  \phi_1(D) dx=\int_{\Om} \nabla \phi_0 \nabla \phi_1dx=\int_\Om \chi_D \phi_0 dx. $$
   Since $-\Delta (u^\ep(D)-\phi_0)=\ep \chi_D$ in $\Om$, we observe that for every $x$ in $\Om$,
   \begin{equation}\label{uepbound}
   \phi_0\leq u^\ep(D)\leq \phi_0 +\ep,
   \end{equation}
   applying  Lemma \ref{uf} -(ii). Defining  $D_\ep=\{x\in \Om:\,\:\phi_0(x)\geq t_M-\ep\}$, we see that 
   \begin{equation}\label{uepondep}
   0\leq u^\ep(D)\leq t_M,\:\:\: on\:\ D_\ep^c\:\:\: and \:\: u^\ep(D)\geq t_M\:\:\: on\:\: D_M,
   \end{equation}
   invoking  \eqref{tmform}, \eqref{uepbound} and Lemma \ref{uf} -(i).
   
   Step $2$. Fixing $\ep>0,$ assume  that $D^*_\ep$ is the maximizer of the optimization problem \eqref{maxp}. We claim that $D^*_\ep$ does not have intersection with
   $D^c_\ep$. Assume $|D^*_\ep\cap D^c_\ep|>0$ an then we have
    \begin{equation}\label{depstar}
   \int_{D^*_\ep} u^\ep(D^*_\ep)dx=\int_{D^*_\ep \cap D_\ep} u^\ep(D^*_\ep)dx+\int_{D^*_\ep\cap D^c_\ep}u^\ep(D^*_\ep)dx<\int_{D^*_\ep \cap D_\ep} u^\ep(D^*_\ep)dx+\int_{B}u^\ep(D^*_\ep)dx,
    \end{equation}
     where $B$ is an arbitrary subset of $D_M\setminus D^*_\ep$ with $|B|=|D^*_\ep \cap D_\ep^c|$.
     The strict inequality holds because of \eqref{uepondep} and the fact that level sets of $u^\ep(D^*)$ have zero measure in view of Lemma \ref{uf} -(iv).
Setting $\tilde{D_\ep}=(D_\ep^* \setminus D_\ep^c)\cup B$, we infer that
$$\int_{\Om} (\bt+\epsilon \chi_{D^*_\ep})  u^\ep(D^*_\ep)dx< \int_{\Om} (\bt+\epsilon \chi_{\tilde{D_\ep}}) u^\ep(D^*_\ep)dx.$$
It is straightforward to deduce
 $\Psi(\ep, D^*_\ep)< \Psi(\ep, \tilde{D_\ep}) $, see the proof of Theorem \ref{incrsecthm}.
This contradicts the optimality of $D^*_\ep$ and thereby proves that $D^*_\ep\subset D_\ep$.

Next, we show that $D_M$ is an approximation for the optimal set $D^*_\ep$ as $\ep$ is small.
We know according to the definition of $D_M$ and $D_\ep$ that $D^*_\ep \setminus D_M\subset D_\ep \setminus D_M $.
If $\ep>0$ tends zero, then $|D^*_\ep- D_M|=|D_M- D^*_\ep|$ tends zero. Since $\|\chi_{D_M}-\chi_{D^*_\ep}\|^2_{L^2(\Om)}=2|D_M-D^*_\ep|$, we deduce
\begin{equation}\label{tendchi}
 \chi_{D^*_\ep}\rightarrow \chi_{D_M}, \quad  as\:\:\:  \ep\rightarrow 0,
    \end{equation}
    in $L^2(\Om)$.
     
      Step $3$. In light of \eqref{tendchi}, we show that there exists $\ep_0>0$ small enough such that for all $0<\ep<\ep_0$ we have $D_M=D^*_\ep$. To do so,
      employing \eqref{psiepform} it is observed that 
      \be \label{diffpsi}
      \Psi(\ep,D_M) -\Psi(\ep,D^*_\ep)=\ep\left(2\int_{\Om} (\chi_{D_M}-\chi_{D^*_\ep})\phi_0 dx+ \ep \int_{\Om} \chi_{D_M} \phi_1(D_M)- \chi_{D^*_\ep} \phi_1(D^*_\ep) dx\right).
      \ee
      Recall that $D_M$ is the maximizer of \eqref{dmmaker} and so
      \be \label{ineqpos}
      \int_{\Om}( \chi_{D_M}-\chi_{D^*_\ep}) \phi_0 dx\geq 0.
      \ee
      We claim that  the second summand in the right hand side of \eqref{diffpsi} converges to zero with higher rate of convergence in comparison with the first summand. We therefore conclude  from \eqref{ineqpos} that there exists $\ep_0$ where
      $$ \Psi(\ep, D_M)\geq \Psi (\ep, D), \:\:\: D\subset \Om,\:\: |D|=A,$$
      for all $0<\ep< \ep_0$. This means that $D_M $ is the maximizer of \eqref{maxp} when $0<\ep< \ep_0$. 
      
       For simplicity we write $\eta_\epsilon= \nabla\phi_1(D_M)-\nabla \phi_1(D^*_\ep)$. Then, we have
       $$\int_{\Om} ( \chi_{D_M}-\chi_{D^*_\ep}) \phi_0 dx=\int_{\Om} \nabla\phi_1(D_M).\nabla\phi_0dx-\int_{\Om} \nabla\phi_1(D^*_\ep).\nabla\phi_0dx=<\eta_\ep, \nabla\phi_0>_{L^2(\Om)},$$
       On the other hand, 
       \begin{align*}
    \ep\left|\int_{\Om} \chi_{D_M}\phi_1(D_M)-\chi_{D^*_\ep}\phi_1(D^*_\ep)dx\right|=\ep\left|\|\phi_1(D_M)\|^2_{H^1_0(\Om)}-\|\phi_1(D^*_\ep)\|^2_{H^1_0(\Om)}\right|=
    \\
    \ep \left(\|\phi_1(D_M)\|_{H^1_0(\Om)}+\|\phi_1(D^*_\ep)\|_{H^1_0(\Om)}\right)\left| \|\phi_1(D_M)\|_{H^1_0(\Om)}-\|\phi_1(D^*_\ep)\|_{H^1_0(\Om)}\right|  \leq
    \\
     \qquad \qquad \qquad\qquad \qquad \qquad C\ep \|\phi_1(D_M)-\phi_1(D^*_\ep)\|_{H^1_0(\Om)}=C\ep<\eta_\ep,\eta_\ep >^{1/2}_{L^2(\Om)},
       \end{align*}
      where we have the inequality because of Lemma \ref{uf} -(i). The generic constant $C$ is independent of $\ep$.
      Recall that  $\chi_{D^*_\ep}\rightarrow \chi_{D_M}$ in $L^2(\Om)$ as $\ep$ tends to zero and so $\eta_\ep \rightarrow 0$ in view of Lemma \ref{uf} -(i).
      
       In summary, we have discovered that the rate of convergence of the first summand  in the right hand side of \eqref{diffpsi} equals to the rate of convergence of the function $<\eta_\ep, \nabla\phi_0>_{L^2(\Om)}$. Moreover,
       the rate of convergence of the second summand  is greater and equal to $\ep<\eta_\ep,\eta_\ep >^{1/2}_{L^2(\Om)}$. Hence, the claim is easily deduced.
       
        Step $4$. In the last step, we will address the  uniqueness of the optimizer. If $D_M \neq D^*_\ep$ for $0<\ep<\ep_0$, then $\int_{\Om} ( \chi_{D_M}-\chi_{D^*_\ep}) \phi_0 dx>0$ since $D_M$ is the  unique solution of  \eqref{dmmaker}.
        This means that the right hand side of \eqref{diffpsi} is positive which yields the uniqueness.

   \end{proof}
It is worth noting here that  solutions of maximization problem \eqref{maxp} are not unique in general \cite{master}. Indeed, Theorem \ref{lowmaxthem} asserts that $f=\alpha\chi_{D_M}+\beta \chi_{D_M^c}$ is the unique maximizer.

 Now, we examine the minimization problem \eqref{minp} in a low contrast regime.  According to Lemma \ref{bathmin}, there is a set $D_m$ which is the  solution  of the following optimization problem
\begin{equation}\label{dm2maker}
\underset{D\in \Om,\: |D|=A}{\inf}\int_{\Om} \chi_D \phi_0dx.
\end{equation}
Note that $D_m $ is uniquely determined by
\begin{equation}\label{tm2form}
D_m=\{x\in \Om:\,\:\phi_0(x)\leq t_m\},\:\: 	t_m=\inf\{s\in \mathbb{R} : |\{x\in \Om:\,\:u(x)\leq s\}|\geq A\}.
\end{equation}
Then, we have the following result for the minimization problem in the low contrast regime which asserts that 
$f=\alpha\chi_{D_m}+\beta \chi_{D_m^c}$ is the unique minimizer of \eqref{minp}.
 \begin{thm}\label{lowminthem}
 	There exists $\ep_0>0$ such that for all $D\subset \Om,\: |D|=A,$ we have
 	$$\Psi(\ep,D_m)\leq \Psi (\ep,D),\quad for \:\: all \:\: 0<\ep\leq\ep_0,$$
 	and the equality occurs only  when $D=D_m$ almost everywhere in $\Om$.
 \end{thm}
\begin{proof}
 The proof is similar to that for Theorem \ref{lowmaxthem} and is omitted.
\end{proof}

It is worth noting here that one can derive optimal sets $D_M$ and $D_m$  easily applying  algorithms like Algorithm 1 in the next section.
\section{High Contrast Regime}\label{high}
 In this section we investigate problems \eqref{maxp} and  \eqref{minp}
 in the hight contrast regime which means that $\al$ and $\bt$ are not close to each other.
 Numerical approaches is developed to determine the solutions of problems  \eqref{maxp} and \eqref{minp}. The algorithms 
 are strongly based on Lemmas \ref{bathmax} and \ref{bathmin} which dealing with level sets of the
 solution $u(x)$ of \eqref{mpde}. We will see that in each iteration steps of our algorithm we need to derive set $\hat{D}$ in 
 Lemma \ref{bathmax} or \ref{bathmin}.
 
   In order to find $\hat{D}$, two algorithms can be developed which they apply the idea of the bisection method. Both algorithms  are the same in essence and we only state the algorithm related to the minimization problem  \eqref{minp}. Introducing the distribution function $F(s)=|\{x\in \Om:\:\: u(x)\leq s \}|$, we state Algorithm $1$ to compute $t$ in Lemma \ref{bathmin} and determine $\hat{D}$ consequently. 
   We should consider a tolerance $TOL$ in the algorithm since it is meaningless computationally to find a set
  $\hat{D}$ satisfying $|D|=A$ exactly.
  \begin{table}[h]\label{bisal}
  	\centering 
  	\begin{tabular}{ l} 
  		\hline 
  		\hline
  		\textbf{Algorithm $1$.} Bisection method for $t$ \\ [1 ex] 
  		\hline 
  		\hline
  		\textbf{Data:} A nonnegative function $u(x)$ where its level sets  have measure zero.\\ 
  		\textbf{Result:} The level $t$.  \\
  		\textbf{$1$.} Set $L= 0, \quad U=\underset{x\in\Om}{\max}\,\, u(x)$;  \\
  		\textbf{$2$.}  Set  $\theta=(L+U)/2$; \\
  		\textbf{$3$.} If $|F(\theta)-A|< TOL$ then set $t=\theta$ \\
  		$\quad$else \\
  		\,\,\,\,\,$\quad$ If $F(\theta)<A$ then \\
  		\,\,\,\,\,$\quad$ $\quad$ Set $L=\theta$; Go to step 2;\\
  		\,\,\,\,\,$\quad$ else\\
  		\,\,\,\,\,$\quad$ $\quad$ Set $U=\theta$; Go to step 2;\\
  		[1ex] 
  		\hline 
  	\end{tabular}
  	\label{table1} 
  \end{table}

  \subsection{Maximization Problem \eqref{maxp}}
   
   Here we describe our algorithm to determine a solution of \eqref{maxp} numerically. We start from a given  vorticity function $f_0$ in $\mathcal{F}$ and extract new  vorticity function $f_1$ such that 
  $\Psi(f_0) \leq \Psi(f_1) $. In this way, we derive a sequence  of functions $\{f_n\}_1^\infty$ such that the corresponding 
  energy functionals $\{\Psi(f_n)\}_1^\infty$ is an increasing sequence of energies.
  \begin{thm}\label{incrsecthm}
  	Assume $f_0=\al\chi_{D_0}+\bt\chi_{D_0^c}$ is a member of $\mathcal{F}$ and $u_{f_0}$ is the solution to
  	\eqref{mpde} with the right hand side $f_0$. Setting $u=u_{f_0}$ in Lemma \ref{bathmax}, suppose that
  	$f_1$ is the maximizer in that lemma. Then $\Psi(f_0) \leq \Psi(f_1) $ and the equality occurs only when
  	$f_0=f_1$ almost everywhere in $\Omega$.
  	
  \end{thm}
  \begin{proof}
  	In what follows, it is convenient to use the following formula for $\Psi(f):$
  	\begin{equation}\label{varu}
  	 \Psi(f)=\int_{\Omega} \left(2fu_f-|\nabla u_f|^2\right) dx=\underset{u\in H^1_0(\Om)}{\sup} \int_{\Omega} \left(2fu-|\nabla u|^2\right) dx.
  	 \end{equation}
  	 In light of Lemma \ref{uf} -(iv), we know that the level sets of $u_{f_0}$ have measure zero. Note that 
  	 \begin{equation}\label{f0leqf1}
  	 \int_{\Omega} f_0 u_{f_0} dx \leq \int_{\Omega} f_1 u_{f_0} dx.
  	 \end{equation}
  	 Invoking \eqref{varu} and \eqref{f0leqf1}, we see that
  	 $$ \Psi(f_1)=\int_{\Omega} \left(2f_1u_{f_1}-|\nabla u_{f_1}|^2\right) dx\geq \int_{\Omega} \left(2f_1u_{f_0}-|\nabla u_{f_0}|^2\right) dx\geq \Psi(f_0). $$
  	 We have  equality in the last expression  if and only if  equality holds in \eqref{f0leqf1}. By the uniqueness assertion in Lemma \ref{bathmax}, we observe that  equality holds in \eqref{f0leqf1} if and only if $f_0=f_1$.
  \end{proof}
  
 Utilizing Theorem \ref{incrsecthm}, we can derive an increasing sequence of energies $\Psi(f_{n-1})\leq \Psi(f_{n})$ with
 starting from a given  $f_0$ in $\mathcal{F}$.
 \begin{thm}\label{maxseqcon}
 	Let $\{f_n\}_1^\infty$ be an increasing sequence derived by Theorem \ref{incrsecthm}. Then, this sequence converges to a local 
 	maximizer of \eqref{maxp}.
 \end{thm}
 \begin{proof}
 	First note that $\|f\|_{\lt}=\|f_1\|_{\lt}$ for all $f$ in $\mathcal{F}$ \cite{bu87,bu89}.
 	Consider the corresponding sequence of energies $\{\Psi(f_n)\}_1^\infty$.
 	The sequence is bounded above since 
 	$$\Psi(f_n)= \int_{\Omega} f_n u_{f_n}dx\leq \|f_n\|_{\lt}\|u_{f_n}\|_{\lt}\leq C \|f_0\|_{\lt} \|u_{f_n}\|_{H^2(\Omega)}\leq C \|f_0\|_{\lt}^2,$$
 	in view of Holder's inequality, Poincar\'{e}'s inequality, and Lemma \ref{uf}. 
 	The sequence $\{f_n\}_1^\infty$ is bounded in $\lt$ then there is a subsequence (still denoted by $\{f_n\}_1^\infty$)
 	converging $\hat{f}$ in $\lt$ with respect to the weak topology. Moreover, $\{u_{f_n}\}_1^\infty$ is a bounded sequence
 	in $H^1_0(\Omega)$ and there is a subsequence (still denoted by $\{u_{f_n}\}_1^\infty$) converging weakly to $\hat{u}$ in
 	$H^1_0(\Omega)$. The compact embedding of $H^1_0(\Omega)$ into $\lt$ (see \cite{brezis}) yields that $\{u_{f_n}\}_1^\infty$ converges strongly to $\hat{u} = u_{ \hat{f} }$ in $\lt$. This leads us to the fact that $\{\Psi(f_n)\}_1^\infty$
 	converges to $\xi=\Psi(\hat{f})$. In view of definition \ref{readef}, it is straightforward to verify that
 	$\bt \leq \hat{f}(x)\leq \al$ in $\Om$. Hence, level sets of $u_{ \hat{f} }$ have measure zero invoking Lemma \ref{uf} -(iv). One can find a member of $\mathcal{F}$ (still denoted by $\hat{f})$ such that $\xi= \Psi(\hat{f})$ because of
 	Lemma \ref{bathmax} and Theorem \ref{incrsecthm}. 
 	
 	So far we have proved that the maximization sequence converges to $\hat{f}$, a member of $\mathcal{F}$. We establish that $\hat{f}$ 	is a local maximizer. Indeed, $\hat{f}=\al\chi_{\hat{D}}+\bt \chi_{\hat{D}}$ where $\{x\in\Omega:\:\:\hat{u}\geq \hat{t} \}$ for $\hat{t}\in \mathbb{R}$. Then, we can introduce increasing function
 	$\phi:\mathbb{R}\rightarrow \mathbb{R} $
 	$$
 	 \phi(t)=
 	 \left\{
 	 \begin{array}{ll}
 	 	\beta\quad\;\;\;\quad t\leq \hat{t} ,
 	 	\\ \al \quad \quad  \;\; \;     t> \hat{t} ,
 	  	 \end{array}
 	 \right.$$
 	 where yields that $\phi(u_{\hat{f}})=\hat{f}$. Therefor, we deduce that $\Psi(\hat{f})\geq \Psi(f)$ for all
 	 $f\in \mathcal{N}$, where $\mathcal{N}$ is a strong neighborhood of $\hat{f}$ relative to $\mathcal{F}$ \cite{ bu89,bu87,bu91,master}. 
 	
 \end{proof}
 
  Now, we provide the details of the maximization algorithm introduced above. At iteration step $n$, there is a guess for the configuration of the optimal vorticity function where it is denoted by $f_n$. We use the finite element method with piecewise linear basis functions to discretize  equation \eqref{mpde} with $f_n$ as its right hand side. 
  
   Let $u_{f_n}$ be a
   solution of \eqref{mpde}  associated with
    $f_n$. For  maximization problem \eqref{maxp},
   we should extract a new  function $f_{n+1}$  based upon the level sets of $u_{f_n}$ where it belongs to $\mathcal{F}$
   and $\Psi(f_n) < \Psi(f_{n+1})$. To derive this $f_{n+1}$,
   we make use of  Lemma  \ref{bathmax} and identify $f_{n+1}$ by setting $u(x)=u_{f_n}$ in that lemma.
   According to Theorems  \ref{incrsecthm} and \ref{maxseqcon}, we have $\Psi(f_n) < \Psi(f_{n+1})$ and the
   generated sequence is convergent. The resulting algorithm is shown in Algorithm 2. There is a stopping criterion in this method.
   The algorithm  stops
   when $\delta \Psi=|\Psi(f_{n+1})-\Psi(f_{n})|$ is less than a prescribed tolerance $TOL$.
   \begin{table}[h]
   	\centering 
   	\begin{tabular}{ l} 
   		\hline 
   		\hline
   		\textbf{Algorithm $2$.} Energy maximization \\ [1 ex] 
   		\hline 
   		\hline
   		\textbf{Data:} An initial  vorticity function $f_0$  \\ 
   		\textbf{Result:} A sequence of increasing energies $\Psi(f_n)$  \\
   		\textbf{$1$.} Set $n = 0$;  \\
   		\textbf{$2$.} Compute $u_{n}$ and $\Psi(f_n)$; \\
   		\textbf{$3$.} Compute $f_{n+1}$ applying Lemma \ref{bathmax}; \\
   		\textbf{$4$.} Compute $\Psi(f_{n+1})$;\\
   	\textbf{$5$.} If $ \delta \Psi< TOL$ then stop;\\
   	$\quad$ else\\
   	$\qquad$ $\qquad$  Set $n=n+1$;\\
   	$\qquad$ $\qquad$ Go to step $2$;\\
   		[1ex] 
   		\hline 
   	\end{tabular}
   	\label{table1} 
   \end{table}
   
    In the third step of Algorithm 2, we should employ Algorithm 1 associated with Lemma \ref{bathmax} to derive set $D_{n+1}$ and then $f_{n+1}=\alpha \chi_{D_{n+1}}+ \bt \chi_{D_{n+1}^c}$.
   \subsection{Minimization Problem \eqref{minp}}
   
    Here we explain our algorithm to derive the global solution of minimization problem \eqref{minp}.    
    In the same spirit as in the maximization case, we initiate from   a given density functions $f_0$ and extract another  vorticity  function $f_1  $ such that the corresponding energy functional decreased , i.e. $\Psi(f_0)\geq \Psi(f_1) $. The problem is that the minimization problem is more complicated and  an iterative method cannot be derived by arguments similar to those in the maximization case. If we consider $f_0$ as an arbitrary  function in  $\mathcal{F}$ and $u_{f_0}$ as an associated solution of  \eqref{mpde},  then
    one can find a density function $f_1$ in $\mathcal{F}$ regarding Lemma \ref{bathmin} such that $\int_{\Om}f_0 u_{f_0} dx \geq \int_{\Om}f_1 u_{f_0} dx$ where the equality occurs only when $f_0=f_1$.
    Then,
    $$\Psi(f_0)= \int_{\Om}2f_0 u_{f_0}- |\nabla u_{f_0}|^2 dx\geq \int_{\Om}2f_1 u_{f_0}- |\nabla u_{f_0}|^2 dx\leq \int_{\Om}f_1 u_{f_1}- |\nabla u_{f_1}|^2 dx=\Psi(f_1).$$
     Hence, we cannot produce a decreasing sequence of energies since the next generated energy may be greater than the previous one.  In  \cite{kao2},   an acceptance rejection method has been added to guarantee a monotone decreasing sequence in such situation. This means that if this new vorticity function decreases the energy, it will be accepted. Otherwise, the partial swapping method will be used. Indeed, in the partial swapping method we use a function $f_1$ in $\mathcal{F}$ where $\delta f =f_1-f_0$ is small enough and $\int_{\Om}f_0 u_{f_0}\geq \int_{\Om}f_1 u_{f_0} $. Particularly, the new function $f_1$  has been determined by trail and error. It is a time consuming task to find the new function $f_1$ such that we should apply an acceptance-rejection method including trail and error.
     In this paper we have improved the procedure for those situations where the function $f_1$ can be selected without the acceptance-rejection method. Indeed, we introduce a criterion to derive the function $f_1$ in a better way.

      The following theorem provides the main tool for energy minimization. Hereafter, we set $\theta= \frac{d}{N{\omega_N}^ {\frac{1}{N}}}$.
      \begin{thm}\label{minseq}
      	Let $f_0=\alpha \chi_{D_0}+\beta \chi_{D_0^c}$ and $f_1=\alpha \chi_{D_1}+\beta \chi_{D_1^c}$ be functions in
      	$\mathcal{F}$. Assume $\int_{\Om} f_0 u_{f_0} dx> \int_{\Om} f_1 u_{f_0} dx$ and $\delta f=f_1-f_0$ be small
      	enough. Then $\Psi(f_0)>\Psi(f_1)$. In particular, if sets $B_1=D_0 \setminus D_1$ and $B_2=D_1 \setminus D_0$ satisfies 
      	\begin{equation}\label{mincon}
      	\int_{\Om} (\chi_{B_2}-\chi_{B_1}) u_{f_0}dx+\theta (\al-\bt)|B_1|^{1+\frac{1}{N}}<0,
            	\end{equation}
            	then $\Psi(f_0)>\Psi(f_1)$.
            	\begin{proof}
            It is easy to see that 
            $$\int_{\Om} (f_0+\delta f) u_{f_0+\delta f}dx= \int_{\Om} f_0  u_{f_0}dx +2 \int_{\Om} \delta fu_{f_0}dx+ \int_{\Om} \delta f u_{\delta f}dx,  $$	
            or
            	\begin{equation}\label{psidiff}	
            	\Psi(f_1)- \Psi(f_0)= 2  \int_{\Om} \delta fu_{f_0}dx+\int_{\Om} \delta f u_{\delta f}dx.
            	            	\end{equation}
On the right hand side of the last equality, we have two integrals where we know that the first one from the left are negative and the second one is positive. We can observe that as $\|\delta f\|_{\lt} \rightarrow 0$, the second integral converges to zero with a higher rate of convergence in comparison with the first integral. Hence, if      $\|\delta f\|_{\lt}$ is small enough, we infer that the right hand side of \eqref{psidiff} is negative and $\Psi(f_0)>\Psi(f_1)$ when  $\|\delta f\|_{\lt} \rightarrow 0$.    

 It is easy to verify that $\delta f= (\al-\bt)\chi_{B_2}+(\bt-\al)\chi_{B_1}$. Note that $|B_1|=|B_2|$ then by \eqref{psidiff} we have  	    
 \begin{align*}
 	\Psi(f_1)- \Psi(f_0)=2 (\al-\bt)\int_{\Om} (\chi_{B_2}-\chi_{B_1}) u_{f_0}dx+ (\al-\bt)^2\int_{\Om} (\chi_{B_2}-\chi_{B_1})
 	(u_{\chi_{B_2}}-u_{\chi_{B_1}})dx=\quad
 	\\
 	 2 (\al-\bt)\int_{\Om} (\chi_{B_2}-\chi_{B_1}) u_{f_0}dx+  (\al-\bt)^2\left ( \int_{\Om}\chi_{B_2}u_{\chi_{B_2}}dx+  \int_{\Om}\chi_{B_1}u_{\chi_{B_1}}dx-2\int_{\Om}\chi_{B_1}u_{\chi_{B_2}}dx\right)<
 	\\ 
 	  2 (\al-\bt)\int_{\Om} (\chi_{B_2}-\chi_{B_1}) u_{f_0}dx+  (\al-\bt)^2\left ( \int_{\Om}\chi_{B_2}u_{\chi_{B_2}}dx+  \int_{\Om}\chi_{B_1}u_{\chi_{B_1}}dx\right)< \qquad  \qquad  \qquad \quad
 	  \\ 2 (\al-\bt)\int_{\Om} (\chi_{B_2}-\chi_{B_1}) u_{f_0}dx+  2(\al-\bt)^2 \theta |B_1|^{1+\frac{1}{N}},  \qquad  \qquad  \qquad \quad 
 	   \qquad     \quad  \qquad  \qquad  \qquad \quad
 \end{align*}    
 in view of Lemma \ref{uf} -(iii). This yields the proof of the second assertion of the theorem. 
            	\end{proof}
      \end{thm}
\begin{rem}\label{decrec}
	If we select $B_1$ and $B_2$ such that $\int_{\Om} \chi_{B_2} u_{f_0}dx < \int_{\Om} \chi_{B_1} u_{f_0}dx $,
	then
	$$\int_{\Om} \delta f  u_{f_0}dx=(\al-\bt) \int_{\Om} (\chi_{B_2}-\chi_{B_1}) u_{f_0}dx<0.$$
	Moreover, $\|\delta f\|_{\lt}=\sqrt{2(\al-\bt)}|B_1|$ will be small if one adjust $|B_1|=|B_2|$ small enough. Then, we have
	$\Psi(f_0)>\Psi(f_1)$ in light of Theorem \ref{minseq}.
	The other way of selecting $B_1$ and $B_2 $ is to find them such that they satisfy \eqref{mincon}.
	
\end{rem}
	
	Utilizing Theorem \ref{minseq} and Remark \ref{decrec}, we can derive a decreasing sequence of energies $\Psi(f_n)\geq \Psi(f_{n+1})$. An interesting result is that the generated sequence converges to the global minimizer of \eqref{minp} with
	any initializer $f_0$ in $\mathcal{F}$.
	\begin{thm}\label{minconev}
	 Let $\{f_n\}_1^\infty$ be a decreasing sequence  derived by Theorem \ref{minseq}. Then, this sequence  converges to the global
	 minimizer of \eqref{minp}.
	\end{thm}
	\begin{proof}
	
	We can prove that the decreasing sequence converges to a local minimizer of \eqref{minp}. The proof is similar to that for Theorem
	\ref{maxseqcon} and is omitted. 
	
	 Assume that $f_1$ and $f_2$ are two local minimizer of \eqref{minp} then
	 \begin{equation}\label{f12con}
	 \int_\Om f_1u_{f_1}dx\leq \int_\Om f_2u_{f_1} \quad and\quad  \int_\Om f_2u_{f_2}dx \leq \int_\Om f_1u_{f_2}dx,
	 \end{equation}
	 in view of Theorem 3.3 of \cite{bu89}. Now, we have
	 \begin{align*}
	 0&\leq&\int_\Om (f_1-f_2)(u_{f_1}-u_{f_2})dx=\int_\Om f_1u_{f_1}-f_1u_{f_2}-f_2u_{f_1}+f_2u_{f_2}dx\leq \qquad\qquad \qquad\qquad
	 \\
	 &\leq&\int_\Om f_2u_{f_1}-f_1u_{f_2}-f_2u_{f_1}+f_1u_{f_2}dx=0,\qquad \qquad\qquad\qquad\qquad \qquad\qquad \qquad\qquad
	 \end{align*}
	 invoking \eqref{f12con}. Hence, we observe that
	 $$\int_\Om (f_1-f_2)(u_{f_1}-u_{f_2})dx=\int_\Om |\nabla u_{f_1-f_2}|^2dx=0,$$
	 which yields that $u_{f_1}=u_{f_2}$ in $\Om$ and then $f_1=f_2$ almost everywhere in $\Om$. Therefor, this says that \eqref{minp} has
	 only one local minimizer which is in fact the unique global minimizer. Consequently, we can deduce that 
	 $\{f_n\}_1^\infty$ converges to the global minimizer.
	\end{proof}
     
       Now, we provide the details of the minimization algorithm introduced above. At iteration step $n$, there is a guess for the configuration of the optimal vorticity function where it is denoted by $f_n=\al \chi_{D_n}+\bt\chi_{D_n^c}$. We use the finite element method with piecewise linear basis functions to discretize  equation \eqref{mpde} with $f_n$ as its right hand side. 
       
       Let $u_{f_n}$ be a
       solution of \eqref{mpde}  associated with
       $f_n$. Based upon level sets of $u_{f_n}$, we should extract new set $D_{n+1}$, $|D_{n+1}|=A$ and then $f_{n+1}=\al \chi_{D_{n+1}}+\bt\chi_{D_{n+1}^c}$ where $\Psi(f_n) > \Psi(f_{n+1})$. Employing Theorem \ref{minseq}, this set is calculated such that it satisfies condition \eqref{mincon} to ensure $\Psi(f_n) > \Psi(f_{n+1})$.
       
       According to Theorem \ref{minconev}, the generated sequence is convergent to the global minimizer.
       The resulting algorithm is shown in Algorithm 3.The stopping criterion is that $\delta \Psi$ should be less than a
       prescribed tolerance $TOL$.
        \begin{table}[h]
        	\centering 
        	\begin{tabular}{ l} 
        		\hline 
        		\hline
        		\textbf{Algorithm $3$.} Energy minimization \\ [1 ex] 
        		\hline 
        		\hline
        		\textbf{Data:} An initial  vorticity function $f_0=\al \chi_{D_0}+\bt\chi_{D_0^c}$  \\ 
        		\textbf{Result:} A sequence of decreasing energies $\Psi(f_n)$  \\
        		\textbf{$1$.} Set $n = 0$;  \\
        		\textbf{$2$.} Compute $D_{n+1}$ applying Lemma \ref{bathmin} with $u=u_{f_n}$, $\al=1,\:\bt=0$; \\
        		\textbf{$3$.}  Set $f_{n+1}=\al \chi_{D_{n+1}}+\bt \chi_{D_{n+1}^c} $; If $\delta \Psi<TOL$, then stop; \\
        		\textbf{$4$.} Set  $B=D_{n+1}\setminus D_{n}$, $B^\prime=D_{n}\setminus D_{n+1}$, $A^\prime=|B|$ ;\\
        		\textbf{$5$.} Compute\\
        		$\quad$ $t=\inf\{s\in \mathbb{R} : |\{x\in B:\,\:u_{f_n}(x)\leq s\}|\geq A^\prime\}$;\\
        		$\quad$ $t^\prime=\sup\{s\in \mathbb{R} : |\{x\in B^\prime:\,\:u_{f_n}(x)\geq s\}|\geq A^\prime\}$;\\
        		\textbf{$6$.} Set $B_2=\{x\in B:\,\:u_{f_n}(x)\leq t\}$, $B_1=\{x\in B^\prime:\,\:u_{f_n}(x)\geq t^\prime\}$;\\
        		\textbf{$7$.} If $B_1$ and $B_2$ satisfy condition \eqref{mincon} then\\
        		$\qquad$  $D_{n+1}=(D_{n} \setminus B_1)\cup B_2$;\\
        		$\qquad$ $D_n=D_{n+1}$;\\
        		$\qquad$ Go to step 2;\\
        		$\quad$ else\\
        		$\qquad$   Set $A^\prime=A^\prime/2$;\\
        		$\qquad$  Go to step $5$;\\
        		[1ex] 
        		\hline 
        	\end{tabular}
        	\label{table1} 
        \end{table}
      
       In the second step of Algorithm $3$, we should invoke Algorithm $1$ in order to compute set $D_{n+1}$. Moreover, we need  Algorithm $1$  to
       calculate parameters $t$ and $t^\prime$ in step $5$ of Algorithm $3$.
       

\begin{figure}[h!]
	\centering
	\subfigure[$\Psi(\hat{D})\approx 13.26$]{\label{max1-cir}\includegraphics[width=0.29\textwidth]{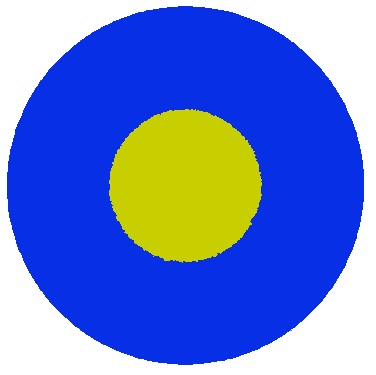}}
	\qquad
	\subfigure[$\Psi(\hat{D})\approx 18.80$]{\label{max1-rec}\includegraphics[width=0.295\textwidth]{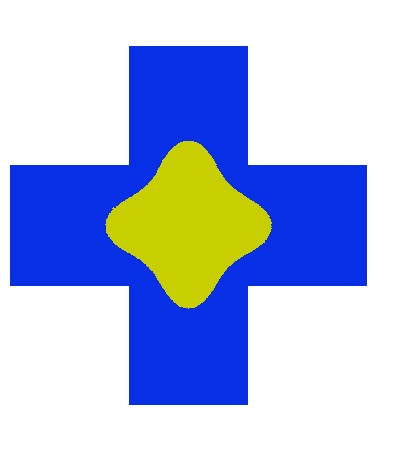}}
	\qquad
	\subfigure[$\Psi(\hat{D})\approx 26.29$]{\label{max1-hea}\includegraphics[width=0.24\textwidth]{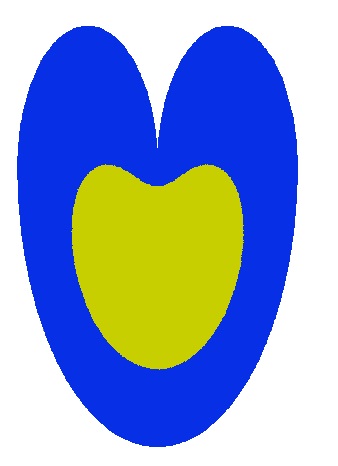}}
	\caption{The maximizer sets in yellow}
	\label{max1}
\end{figure}
\begin{figure}[h!]
	\centering
	\subfigure[$\Psi(\hat{D})\approx 7.14$]{\label{min-cir}\includegraphics[width=0.29\textwidth]{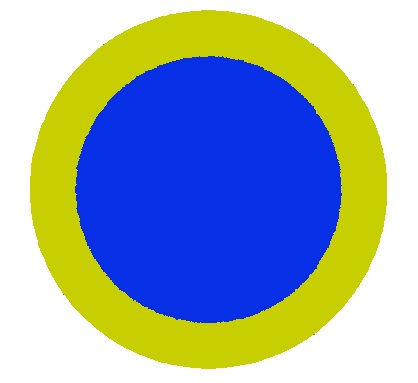}}
	\qquad
	\subfigure[$\Psi(\hat{D})\approx 8.78$]{\label{min-rec}\includegraphics[width=0.29\textwidth]{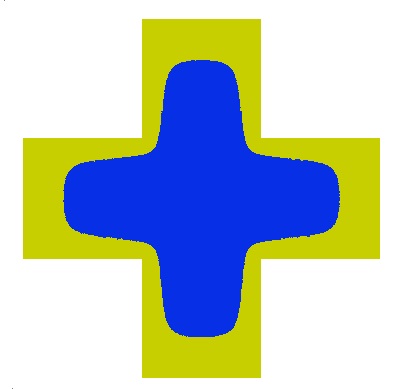}}
	\qquad
	\subfigure[$\Psi(\hat{D})\approx 12.29$]{\label{min-hea}\includegraphics[width=0.24\textwidth]{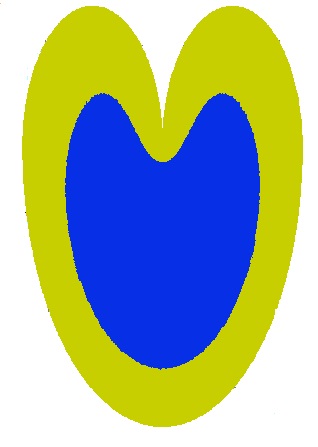}}
	\caption{The minimizer sets in yellow}
	\label{min}
\end{figure}

\begin{figure}[h!]
	\centering
	\subfigure[$\Psi(\hat{D})\approx 2.19 $]{\label{dumble-max-u}\includegraphics[width=0.18\textwidth]{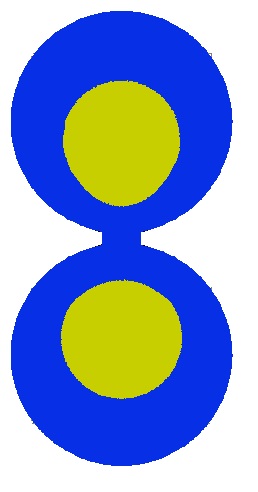}}
	\qquad
	\subfigure[$\Psi(\hat{D})\approx 1.11$]{\label{dumble-min-u}\includegraphics[width=0.2\textwidth]{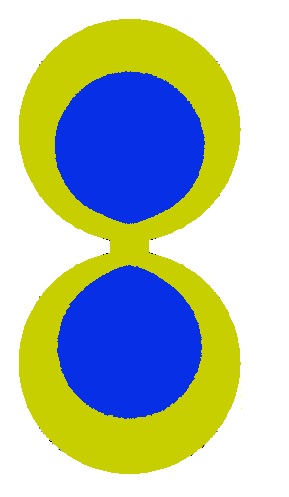}}
	\qquad
	\subfigure[$\Psi(\hat{D})\approx 4.72$]{\label{dumble-max-a-loc}\includegraphics[width=0.19\textwidth]{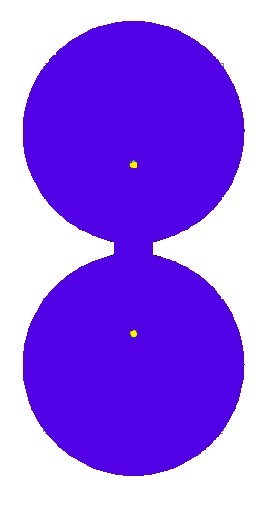}}
	\qquad
		\subfigure[$\Psi(\hat{D})\approx 6.30$]{\label{dumble-max-a-gol}\includegraphics[width=0.185\textwidth]{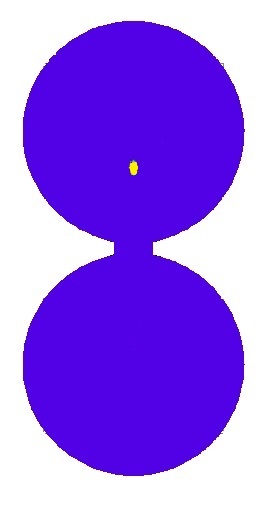}}
	\caption{The optimizer sets in yellow}
	\label{dumble}
\end{figure}

\section{Implementing The Algorithms}\label{implement}
  
 In this section some examples are chosen to illustrate the numerical solutions of the optimization problems \eqref{maxp} and \eqref{minp} with an eye on the physical interpretation of them. We present some results in dimension $N=2$ based on Algorithms  $1, \:2$ and $3$.
 
  In the following examples, we have set  $TOL=5\times 10^{-3}$ for  Algorithms  $1, \:2$ and $3$. These algorithms  typically converge in less than ten iterations for all examples.
  \\
  \\
  \textbf{Example 1.}  Setting $\alpha=2$ and $\beta=1$, we want to find  the solution of the maximization problem \eqref{maxp} using Algorithm 2. We illustrate
   the optimal set when $\Omega$ has different shapes. Remember that our aim is to determine a set $\hat{D}\subset \Omega$ so to maximize the
   energy functional $\Psi(D)$. Such optimal sets are plotted in Figure \ref{max1} for various geometries $\Omega$. The maximizer sets $\hat{D}$ are depicted in yellow. In these numerical experiments we set $|A|=\pi ,\: |\Omega|=4\pi$ in Figure \ref{max1-cir},  $|A|=6 ,\: |\Omega|=20$ in Figure \ref{max1-rec} and $|A|= 7.80,\: |\Omega|=18.85$ in Figure \ref{max1-hea}. 
 \\
 \\
  \textbf{Example 2.}  In this example we want to find the solution of the minimization problem \eqref{minp} invoking Algorithm 3. Again, we have set     $\alpha=2$ and $\beta=1$. We show
  the optimal set when $\Omega$ has different geometries.  The yellow sets are the minimizer sets $\hat{D}$.  Parameters $A$, $|\Omega|$ are same as the previous example. 
  
   Let us take a look at physical interpretation of the solutions. If $f$ is an optimal solution derived in the above examples, then there is a monotone function $\phi:\mathbb{R}\rightarrow \mathbb{R}$ where
   $f=\phi(u_f)$ \cite{bu87,bu89}. Then, $u_f$ represents the stream function for the steady flow of an ideal fluid in two dimensions confined by a solid wall in the shape $\Omega$. The velocity field is given by $(u_{x_2},-u_{x_1})$ and the vorticity which is given by
   the curl of the velocity has magnitude $f=-\Delta u_f$. The optimal solutions show different configurations of a region of vorticity
   in an otherwise irrotational flow. They correspond to stationary and stable flows.
     \\
  \\
  \textbf{Example 3.} In this example we show that there is a  drawback for the maximization algorithm. Such algorithms may
  stick to a local optimizers. To overcome this problem a typical method is to run the maximization algorithm with different initializers. Then one   can compare the derived maximizers and choose the best one. 
  In this example the domain is a dumbbells with $|\Omega|=6.32$. In Figures \ref{dumble-max-u} and \ref{dumble-min-u} we have tested Algorithms 2 and 3  with parameters $A=2.54,\: \alpha=2,\: \beta=1$.  
   Figure \ref{dumble-min-u} illustrates the global minimizer set in yellow. It is noteworthy that the minimization Algorithm 3 converges
   to the global minimizer from any initializer.
  Starting from different initial sets, our numerical tests converge to the maximizer set, the yellow set, in Figure \ref{dumble-max-u}. 
  Although this has not been established theoretically, the numerical tests converge to the  global maximizer set of the respective problem. 
  Figure \ref{dumble-max-a-loc} and  Figure \ref{dumble-max-a-gol} show the maximizer sets employing parameter $A=0.03,\: \alpha=100,\: \beta=1$. Indeed, $\alpha$ and $\beta$ are not close to each other, the high-contrast regime. If one initiates Algorithm 2 from a set
  where intersects both lobes, typically the optimal set determined by the algorithm will be the yellow set in Figure \ref{dumble-max-a-loc}.
  Indeed, this set is a local maximizer where the algorithm  sticks to it.
  On the other hand, if we run the algorithm from an initializer set in one of the lobes, the algorithm converges to a maximizer set in  the lobe
  which we have started from it, see Figure \ref{dumble-max-a-gol}. This is the global maximizer of the problem and it reveals that
  this problem has two global maximizers.
  
  Example 3 reveals that a  dumbbells domain has three local maximizers. This is in accordance with the results 
  of \cite{master, bu89}. Indeed, our algorithm is capable of deriving local maximizers of complicated domains such as
  domain $\Omega$ that approximates the union of $n$ balls.
  

\section{Conclusions}

 In this paper, we have studied two optimization problems associated with   Poisson boundary value problem \eqref{mpde}.
 An optimal solution $f$ that maximize or minimize optimization problems \eqref{maxp} or \eqref{minp} 
in a set of rearrangements  defines a stationary and stable flow of an ideal fluid in two dimensions, confined by a solid wall in the
shape of $\Omega$. 

 Although  there is a plenitude of papers studied rearrangement optimization problems \eqref{maxp} and \eqref{minp} including investigation of 
the  existence, uniqueness and some qualitative properties of the solutions, 
 the question of the exact formula of the optimizers or  optimal shape design have  been considered just in case that $\Omega$ is a ball \cite{bu91}.

 The main contribution of this paper is investigating and determining the optimal shape design for a general domain $\Omega$. At first, we have addressed this question analytically when the problem is in  low contrast regime. Although it has been proved that solutions of \eqref{maxp}
 are not unique in general, we have established that the  solutions of both problems \eqref{maxp} and \eqref{minp} are unique when $\alpha$ and $\beta$ are close to each other. Indeed, the analytical solutions of  \eqref{maxp} and \eqref{minp} is determined by a super-level set or
 sub-level set of the solution of \eqref{mpde} with right-hand side $\beta$.
  
  Secondly,  when $\al$ and $\bt$ are not close to each other, the high contrast regime, there must be numerical approaches to determine the optimal shape design. Two optimization algorithms have been developed in order to find the optimal energies for problems \eqref{maxp} and \eqref{minp}. For the minimization problem, we have proved that our algorithm converges to the global minimizer of \eqref{minp} regardless of the initializer.
 In our algorithm for the minimization problem \eqref{minp}, we have replaced the partial swapping method used in \cite{kao2} with a step where one should only verify that some of data fulfill a criterion. Partial swapping method increases the cost of computations since one must accept or reject some data by checking whether the objective function is improved or not.
 
 For the maximization problem \eqref{maxp}, an algorithm have been developed where we have proved that it converges to a local maximizer.
 Running the algorithm with different initializers, one can obtain the global maximizer.
 Particularly, our algorithm is capable of deriving all local maximizers including the global one  for complicated domains such as
 domain $\Omega$ that approximates the union of $n$ balls. 
 
 Setting $\hat{f}$ as the global minimizer of the problem \eqref{minp} derived by Algorithm 3, numerical tests in the previous section reveal that  the local maximizers are in the farthest away from $\hat{f}$ relative to $\mathcal{F}$ . Since $\|f\|_{L^2(\Om)}=\|f_0\|_{L^2(\Om)}$ for all  $f\in \mathcal{F} $ \cite{bu89}, then we see
 \begin{equation*}
\|f-\hat{f}\|^2_{L^2(\Om)}=2\|f_0\|^2_{L^2(\Om)}-2\int_{\Om}f \hat{f}dx,
 \end{equation*}
 and so solutions of the minimization problem
 \begin{equation}\label{mineq}
\underset{ f\in \mathcal{F}}{\min} \int_{\Om}f \hat{f}dx,
 \end{equation}
are in the farthest away from $\hat{f}$ relative to $\mathcal{F}$.
Problem \eqref{mineq} has a solution  \cite{bu87,bu89} and the solution can be calculated by Algorithm 1.
All maximizers in  examples 2 and 3 are solutions of \eqref{mineq}. We conjecture that  all local maximizers are solutions of
\eqref{mineq}. Then we can recast optimization problem \eqref{maxp}  in
\begin{equation*}
\underset{f\in \mathcal{F^\prime}}{\max } \Psi(f),
\end{equation*}
where $\mathcal{F^\prime} \subset \mathcal{F} $ is the  solutions of \eqref{mineq}. It seems that solving this equation numerically is simpler since we are searching for
the maximizer in a smaller set of functions.
Our maximization algorithm will converge faster if one starts from a member of $\mathcal{F^\prime}$.
 It would be interesting if one studies this new maximization problem analytically or numerically since
  local maximizers are also corresponding to steady flows of the fluid and deriving them are physically important. 



\end{document}